\documentclass{elsarticle}

\usepackage[latin1]{inputenc}
\usepackage{amsmath}
\usepackage{amssymb}
\usepackage{amsthm}

\usepackage{multirow}
\usepackage{bigdelim}

\allowdisplaybreaks[3]

\theoremstyle{plain}
\newtheorem{theorem}{Theorem}[section]
\newtheorem{lemma}[theorem]{Lemma}

\theoremstyle{definition}

\numberwithin{equation}{section}

\newcommand\D{{\mathcal D}}

\newcommand\RR{{\mathbb R}}

\newcommand\NN{{\mathbb N}}

\newcommand*\pFqskip{8mu}
\catcode`,\active
\newcommand*\pFq{\begingroup
        \catcode`\,\active
        \def ,{\mskip\pFqskip\relax}%
        \dopFq
}
\catcode`\,12
\def\dopFq#1#2#3#4#5{%
        {}_{#1}F_{#2}\biggl(\genfrac..{0pt}{}{#3}{#4};#5\biggr)%
        \endgroup
}

\begin{document}

\begin{frontmatter}

\title{From Krall discrete orthogonal polynomials to Krall polynomials%
\tnoteref{nota}%
}
\tnotetext[nota]%
{Partially supported by MTM2015-65888-C4-1-P  (Ministerio de Econom\'ia y Competitividad),
FQM-262, FQM-7276 (Junta de Andaluc\'ia) and Feder Funds (European
Union).}

\author{Antonio J. Dur\'an}
\address{Departamento de An\'alisis Matem\'atico, Universidad de Sevilla, 41080 Sevilla, Spain}
\ead{duran@us.es}

\begin{abstract}
We show how to get Krall polynomials from Krall discrete polynomials using a procedure of passing to the limit in some of the parameters of the family. We also show that this procedure has to be different to the standard one used in the Askey scheme to go from the classical discrete
 polynomials to the classical polynomials.
\end{abstract}


\maketitle
\end{frontmatter}

\section{Introduction}
The hypergeometric representation of classical and classical discrete orthogonal polynomials gives rise to the socalled Askey scheme or Askey tableau (see \cite{KLS}, Ch. 9). Askey scheme is an organized hierarchy of hypergeometric orthogonal polynomials, namely, Hermite, Laguerre, Jacobi (the classical families), Charlier, Meixner, Krawtchouk, Hahn (the classical discrete families), dual Hahn, Racah, Wilson, etc.
One can navigate through the Askey scheme by passing to the limits in the parameters. For instance, one can get the Laguerre polynomials $(L_n^\alpha)_n$, from the Meixner polynomials $(m_n^{a,c})_n$ (normalized so that the leading coefficient of $m_n^{a,c}$ is $1/n!$) by setting $x\to x/(1-a)$ and $c=\alpha +1$ and taking limit as $a\to 1$; more precisely
\begin{equation}\label{blmel}
\lim_{a\to 1}(a-1)^nm_n^{a,\alpha +1}\left(\frac{x}{1-a}\right)=L_n^{\alpha}(x)
\end{equation}
see \cite{KLS}, p. 243 (take into account that we are using for the
Meixner polynomials a different normalization to that in \cite{KLS}).

Krall or bispectral polynomials are one of the most interesting generalization of classical and classical discrete polynomials.
Krall polynomials are orthogonal polynomials with respect to a positive measure which in addition are eigenfunctions of a differential operator of order $k$ bigger than 2. Krall discrete polynomials appear by changing in that definition differential operator to difference operator.

Krall polynomials were introduced by H.L. Krall. In the late thirties, he proved that orthogonal polynomials only can be eigenfunctions of
a differential operator of even order. He also
classified all families of orthogonal polynomials which in addition are eigenfunctions of a fourth order differential operator (see \cite{Kr1,Kr2}). Besides the classical families, he found three other families of polynomials orthogonal with respect to the positive measures $e^{-x}+u\delta_0$, $x\ge 0$, $(1-x)^\alpha +u\delta_1$, $-1< x <1$, and $1+u\delta_{-1}+u\delta_1$, $-1\le x \le 1$, respectively.

Since the eighties, a big amount of effort has been devoted to this issue (with contributions
by L.L. Littlejohn \cite{L1,L2,L3}, A. M. Krall  \cite{KrL}, J. Koekoek and R. Koekoek \cite{koekoe,koe,koekoe2},
F.A. Grünbaum and L. Haine et al. \cite{GrH,GrHH,GrY}, A. Zhedanov \cite{Z}, K.H. Kwon et al. \cite{KLY,KLY2}, P. Iliev \cite{Il1,Il2},
and the list is by no means exhaustive).

If we consider positive measures, the recipe for constructing Krall polynomials is to take the Laguerre and Jacobi weights, assume that one or two of the parameters are nonnegative integers and add a Dirac delta at one or two of the end points of the interval of
orthogonality. For the Laguerre case, the precise result is the following (see \cite{koekoe}; also \cite{Il2},\cite{du1}).
For $\kappa$ a positive integer, the orthogonal polynomials $(L_n^{\kappa;u})_n$ with respect to the positive measure
\begin{equation}\label{lagk}
x^{\kappa-1} e^{-x}+\kappa!u\delta_{0},
\end{equation}
are eigenfunctions of a differential operator of order $2\kappa+2$. Orthogonal polynomials $(L_n^{\kappa;u})_n$ with respect to (\ref{lagk}) can be expanded in terms of two consecutive Laguerre polynomials as follows (see Lemma A.3 in \cite{du1}; also \cite{Il2}. For other representations see \cite{koor})
\begin{equation}\label{elagp}
L_n^{\kappa;u}(x)=L_n^{\kappa}(x)-\frac{(\kappa-1)!+u(n+1)_{\kappa}}{(\kappa-1)!+u(n)_{\kappa}}L_{n-1}^{\kappa}(x).
\end{equation}
The story of Krall discrete polynomials is somehow different. In fact, except for the discrete classical families themselves, no examples satisfying difference equations of finite order bigger than two were known until very recently. What we knew for the continuous case (or for the $q$ case) had seemed to be of little help  because adding Dirac deltas to the discrete classical families
seems not to work. Indeed, answering a question posed by R. Askey in 1991 (see page 418 of \cite{BGR}), H. Bavinck, H. van Haeringen and R. Koekoek
studied orthogonal polynomials with respect to Charlier and Meixner measures together with a Dirac delta at $0$ and found that they
satisfy certain type of difference equations of infinite order (\cite{BH,BK}). In 2012, this author introduced
the first examples of orthogonal polynomials $(p_n)_n$ which are common eigenfunctions of a difference
operator of order bigger than $2$ and posed a number of conjectures about them (\cite{du0}). The orthogonalizing measures for these families of polynomials are constructed by multiplying the classical discrete weights by carefully chosen polynomials. This kind of transformation which consists in multiplying a measure $\mu$ by a polynomial $q$ is called a Christoffel transform. It has a long tradition in the context of orthogonal polynomials: it goes back a century and a half ago when
E.B. Christoffel (see \cite{Chr} and also \cite{Sz}) studied it for the particular case $q(x)=x$. The conjectures posed in \cite{du0} have already been proved in a number of papers (\cite{du1,ddI,ddI2}, see also \cite{dudh}). Here it is an example constructed from the Meixner measure.

For $\kappa$ a positive integer and $a$ and $c$ real numbers with $0<a<1$ and $c>\kappa+1$, the orthogonal polynomials $(m_n^{a,c;\kappa})_n$ with respect to the positive measure
\begin{equation}\label{mk}
\sum_{x=0}^\infty \prod_{j=1}^{\kappa} (x+c-j)\frac{a^x\Gamma(x+c-\kappa-1)}{x!}\delta_x,
\end{equation}
are eigenfunctions of a difference operator of order $2\kappa+2$.  Orthogonal polynomials $(m_n^{a,c;\kappa})_n$ with respect to (\ref{mk}) can be expanded in terms of two consecutive Meixner polynomials as follows (see Theorem 5.4 in \cite{du1})
\begin{equation}\label{emp}
m_n^{a,c;\kappa}(x)=m_n^{a,c}(x)+\frac{am_{\kappa}^{1/a,2-c}(-n-1)}{(1-a)m_{\kappa}^{1/a,2-c}(-n)}
m_{n-1}^{a,c}(x).
\end{equation}
The positivity of the measure (\ref{mk}) implies that $m_{\kappa}^{1/a,2-c}(-n)\not =0$, $n\in \NN$.

Taking into account the limit (\ref{blmel}) in the Askey scheme, a natural question arises: can one get in the same form the Krall polynomials $(L_n^{\kappa;u})_n$ (\ref{elagp}) from the Krall discrete polynomials $(m_n^{a,c;\kappa})_n$ (\ref{emp}). The answer is no. Indeed, in Section 2 we prove that
\begin{equation}\label{blmeli}
\lim_{a\to 1}(a-1)^nm_n^{a,c;\kappa}\left(\frac{x}{1-a}\right)=\begin{cases} L_n^{c-2}(x),& \mbox{for $c\not =2,\cdots, \kappa +1$},\\
L_n^{c-1;\infty}(x),& \mbox{for $c=2,\cdots, \kappa +1$},
\end{cases}
\end{equation}
where $L_n^{\kappa;\infty}(x)=L_n^{\kappa}(x)-\frac{n+\kappa}{n}L_n^{\kappa}(x)$, which correspond with the degenerate case of (\ref{elagp}) for $u=\infty$. If the parameters do not satisfy the assumptions $0<a<1$ and $c>\kappa+1$, neither the positivity nor the existence of the measure (\ref{mk}) is guaranteed (this is the case when $c=2,\cdots, \kappa +1$); in that case we  implicitly assume
that $m_{\kappa}^{1/a,2-c}(-n)\not =0$, $n\in\NN$, so that the polynomials (\ref{emp}) are well-defined.

The purpose of this paper is to show that one can get the Krall polynomials from the Krall discrete polynomials, but taking limits in a different way to the procedure used in the Askey scheme. There one takes limit in only one parameter of the classical discrete family while the other parameters are fixed and independent of the parameter in which one takes limit. For Krall polynomials, one has to carefully choose suitable functions and equal all those parameters to such functions in the parameter in which one takes limit. As an example, here it is a way to get the Krall Laguerre polynomials $(L_n^{\kappa;u})_n$ (\ref{elagp}) from the Krall Meixner polynomials $(m_n^{a,c;\kappa})_n$ (\ref{emp}): in this case we choose the function $\phi(x)=\kappa+1+(1-x)^{\kappa}/u$ and write $c=\phi(a)$. More precisely
\begin{equation}\label{limit1}
\lim_{a\to 1}(a-1)^nm_n^{a,\kappa+1+(1-a)^{\kappa}/u;\kappa}\left(\frac{x}{1-a}\right)=L_n^{\kappa;u}(x).
\end{equation}
The limits from Krall Meixner to Krall Laguerre polynomials will be considered in Section 2, while in Section 3, we study the limits from Krall Hahn to Krall Jacobi polynomials.

We finish this introduction with a comment on exceptional polynomials. Exceptional orthogonal polynomials $p_n$, $n\in X\varsubsetneq \NN$, are complete orthogonal polynomial systems with respect to a positive measure which in addition
are eigenfunctions of a second order differential operator. Exceptional discrete polynomials appear by changing in that definition differential operator to difference operator. Exceptional and exceptional discrete polynomials extend the  classical families of Hermite, Laguerre and Jacobi,
and the classical discrete families of Charlier, Meixner, Krawtchouk and Hahn, respectively.
In mathematical physics, exceptional polynomials allow to write exact
solutions to rational extensions of classical quantum potentials.
The last few years have seen a great deal of activity in the area  of exceptional orthogonal polynomials, mainly by theoretical physicists (see, for instance,
\cite{duch,dume,duhj,GUKM1,GUKM2} (where the adjective \textrm{exceptional} for this topic was introduced), \cite{GUGM,G,OS0,OS4,Qu}, and the references therein).

The most apparent difference between classical polynomials and exceptional polynomials
is that the exceptional families have gaps in their degrees, in the
sense that not all degrees are present in the sequence of polynomials (as it happens with the classical families). This
means in particular that they are not covered by the hypotheses of Bochner's and Lancaster's classification theorems for classical and classical discrete polynomials, respectively (\cite{B},\cite{La}).

Up to now no relationship has been found between Krall
and exceptional polynomials (what it is understandable because they are coming from very different problems).
However, the situation is completely different at the discrete level, where an unexpected connection between Krall discrete and exceptional discrete polynomials has been found (\cite{duch,dume,duhj}). The key for this connection is a well-known and fruitful concept regarding discrete orthogonal polynomials: duality (between the variable $x$ and the index $n$ of the polynomial). It has been shown that duality interchanges exceptional discrete polynomials with Krall discrete polynomials.
What it is relevant for this paper is that one can construct exceptional polynomials from exceptional discrete polynomials
by taking limits in some of the parameters in the same way as one goes from classical
discrete polynomials to classical polynomials in the Askey tableau. However, as we show here, despite exceptional discrete and Krall discrete polynomials are dual, this procedure does not work when dealing with Krall and Krall discrete polynomials, and it has to be changed as explained in this paper.

\section{From Krall Meixner to Krall Laguerre polynomials}
In this Section we study how to get the Krall Laguerre polynomials $(L_n^{\kappa;u})_n$ orthogonal with respect to the positive measure (\ref{lagk}) from the Krall Meixner polynomials $(m_n^{a,c;\kappa})_n$ orthogonal with respect to the positive discrete measure (\ref{mk}).

We start with some basic definitions and facts about Meixner and Laguerre polynomials, which we will need later.
For $a\not =0, 1$ we write $(m_{n}^{a,c})_n$ for the sequence of Meixner polynomials defined by
\begin{align}\label{mxpol}
m_{n}^{a,c}(x)&=\frac{a^n}{(1-a)^n}\sum _{j=0}^n a^{-j}\binom{x}{j}\binom{-x-c}{n-j}\\\label{mxpol2}&=\frac{a^n(c)_n}{(a-1)^nn!}\pFq{2}{1}{-n,-x}{c}{1-1/a}
\end{align}
(we have taken a slightly different normalization from the one used in \cite{KLS}, pp, 234-7).
They satisfy the following three term recurrence formula ($m_{-1}=0$)
\begin{equation}\label{mxttrr}
xm_n^{a,c}=(n+1)m_{n+1}^{a,c}-\frac{(a+1)n+ac}{a-1}m_n^{a,c}+\frac{a(n+c-1)}{(a-1)^2}m_{n-1}^{a,c}, \quad n\ge 0.
\end{equation}
For $\alpha\neq-1,-2,\ldots$, we write $(L_n^\alpha )_n$ for the sequence of Laguerre polynomials
\begin{equation}\label{deflap}
L_n^{\alpha}(x)=\sum_{j=0}^n\frac{(-x)^j}{j!}\binom{n+\alpha}{n-j}
\end{equation}
(that and the next formulas can be found in \cite{EMOT}, vol. II, pp. 188--192; see also \cite{KLS}, pp, 241-244).
They satisfy the three-term recurrence formula ($L_{-1}^{\alpha}=0$)
\begin{equation}\label{lagttrr}
xL_n^{\alpha}=-(n+1)L_{n+1}^{\alpha}+(2n+\alpha+1)L_n^{\alpha}-(n+\alpha)L_{n-1}^{\alpha}.
\end{equation}
We also use the well-known formula
\begin{equation}\label{f1lag}
L_n^{\alpha-1}(x)=L_n^{\alpha}(x)-L_{n-1}^{\alpha}(x).
\end{equation}
From (\ref{mxpol2}) we easily see
\begin{align}\label{lmk}
\lim_{a\to 1}(a-1)^\kappa m_\kappa^{1/a,2-c}(z)&=(-1)^\kappa \binom{\kappa+1-c}{\kappa}, \quad c\not=2,\cdots, \kappa+1,\\
\label{lmk2}
\lim_{a\to 1}(a-1)^{\kappa+1-c} m_\kappa^{1/a,2-c}(z)&=(-1)^{\kappa+1-c} \binom{z}{c-1}, \quad c=2,\cdots, \kappa+1.
\end{align}
Using induction on the recurrence relations (\ref{mxttrr}) and (\ref{lagttrr}), one can easily generalize the
limit (\ref{blmel})  as follows: if $\phi$ is a function of $a$ with
$\lim_{a\to 1}\phi(a)=\alpha+1$ then
\begin{equation}\label{blmel3}
\lim_{a\to 1}(a-1)^nm_n^{a,\phi(a)}\left(\frac{x}{1-a}\right)=L_n^{\alpha}(x).
\end{equation}
\bigskip

We first prove the limits (\ref{blmeli}). The first one is an easy consequence of (\ref{emp}), (\ref{blmel}), (\ref{lmk}) and (\ref{f1lag}).
The second one is straightforward from (\ref{emp}), (\ref{blmel}), (\ref{lmk2}) and (\ref{elagp}).

\bigskip

We next prove the limit (\ref{limit1}). For the benefit of the reader we consider first the case $\kappa=1$.
The numerator in the coefficient of $m_{n-1}^{a,c}$ in (\ref{emp}) is $m_1^{1/a,2-c}(-n-1)$. Using (\ref{mxpol}), we  have
$$
m_1^{1/a,2-c}(x)=x+\frac{c-2}{a-1}.
$$
We see then that if we write $c=\phi(a)$ with $\phi(1)=2$, we have
$$
\lim_{a\to 1}\frac{c-2}{a-1}=\phi'(1),
$$
and $\phi'(1)$ can be used as a free parameter. Hence, by setting $c=2+(1-a)/u$, we straightforwardly get using (\ref{emp}),
(\ref{blmel3}) and (\ref{elagp}):
\begin{equation*}
\lim_{a\to 1}(a-1)^nm_n^{a,2+(1-a)/u;1}\left(\frac{x}{1-a}\right)=L_n^{1;u}(x).
\end{equation*}
In order to manage the case for an arbitrary positive integer $\kappa$, we need the following Lemma.
\begin{lemma}
$$
\lim_{a\to 1}m_{\kappa}^{1/a,-k+1-(1-a)^{\kappa}/u}(z)=\frac{(-1)^{\kappa}}{\kappa u}+\frac{(z-\kappa+1 )_{\kappa}}{\kappa !}.
$$
\end{lemma}

\begin{proof}
It follows from (\ref{mxpol2}) after a careful computation.
\end{proof}

The limit (\ref{limit1}) follows now easily using the previous Lemma, (\ref{emp}),
(\ref{blmel3}) and (\ref{elagp}).

\bigskip
For $\kappa$ a positive integer and $a$ and $c$ real numbers with $0<a<1$ and
$$
\begin{cases} c>\kappa+1,&\mbox{for $\kappa$ even},\\
\displaystyle c\in\cup_{j=-(\kappa-1)/2+1}^1(\kappa +2j-2,\kappa+2j-1),&\mbox{for $\kappa$ odd},
\end{cases}
$$
consider now the positive measure
\begin{equation}\label{mk2}
\sum_{x=-\kappa-1}^\infty \prod_{j=1}^{\kappa} (x+j)\frac{a^x\Gamma(x+c)}{(x+\kappa+1)!}\delta_x.
\end{equation}
Orthogonal polynomials $(M_n^{a,c;\kappa})_n$ with respect to (\ref{mk2}) can be expanded in terms of two consecutive Meixner polynomials as follows (see Theorem 5.7 in \cite{du1})
\begin{equation}\label{emp2}
M_n^{a,c;\kappa}(x)=m_n^{a,c}(x)+\frac{m_{\kappa}^{a,2-c}(-n-1)}{(1-a)m_{\kappa}^{a,2-c}(-n)}
m_{n-1}^{a,c}(x).
\end{equation}
The polynomials $(M_n^{a,c;\kappa})_n$ are also eigenfunctions of a difference operator of order $2\kappa+2$ (see \cite{du1}).
Proceeding as before one can also get the Krall Laguerre polynomials (\ref{elagp}) from the Krall Meixner polynomials
$(M_n^{a,c;\kappa})_n$ (\ref{emp2}):
\begin{equation}\label{limit2}
\lim_{a\to 1}(a-1)^nM_n^{a,\kappa+1+(a-1)^{\kappa}/u;\kappa}\left(\frac{x}{1-a}\right)=L_n^{\kappa;u}(x).
\end{equation}

\section{From Krall Hahn to Krall Jacobi polynomials}
In this section we study the limits between Krall Hahn and Krall Jacobi polynomials.

We include here basic definitions and facts about dual Hahn, Hahn and Jacobi polynomials, which we will need later.

For a real number $u$ and a positive integer $j$, we consider the polynomials $\lambda^{u}_j$ and $\theta^u$ of degree $j$ and $2$, respectively, defined by
\begin{align}\label{deflamb}
\lambda^{u}_j(x)&=\prod_{i=0}^{j-1}(x-i(u+i+1)),\\\label{defth}
\theta_x^u&=x(x+u+1).
\end{align}
For real numbers $a $ and $b$, we write $(R_{n}^{a,b,N})_n$ for the sequence of dual Hahn polynomials defined by
\begin{equation}\label{dhpol}
R_{n}^{a,b,N}(x)=\sum _{j=0}^n\frac{(-n)_j(-N+j)_{n-j}(a+j+1)_{n-j}}{(-1)^j j!}\lambda^{a+b}_j(x)
\end{equation}
(we have taken a slightly different normalization from the one used in \cite{KLS}, pp, 234-7 from where the next formulas can be easily derived). Notice that $R_{n}^{a,b,N}$ is always a polynomial of degree $n$.
Using that
\begin{equation}\label{lth}
(-1)^j\lambda^{a+b}_j(\theta_x^{a+b})=(-x)_j(x+a+b+1)_j,
\end{equation}
we get the hypergeometric representation
$$
R_{n}^{a,b,N}(\theta_x^{a+b})=(-N)_n(a+1)_n\pFq{3}{2}{-n,-x,x+a+b+1}{a+1,-N}{1}.
$$
The following limits are consequences of (\ref{dhpol}) (after careful computations)
\begin{align}\label{lhk}
\lim_{N\to \infty}\frac{R_\kappa^{-b,-a,a+b+N}(z)}{N^\kappa}&=(b-\kappa)_\kappa, \quad b\not=1,\cdots, \kappa,\\
\label{lhk2}
\lim_{N\to \infty}\frac{R_\kappa^{-b,-a,a+b+N}(z)}{N^{\kappa-b}}&=\frac{(b+1)_{\kappa-b}}{(-1)^{b+\kappa}}
\lambda_b^{-b-a}(z), \quad b=1,\cdots, \kappa,\\\label{lhk3}
\lim_{N\to \infty}R_\kappa^{-\phi_\kappa^s(N),-\psi(N),\psi(N)+\phi_\kappa^s(N)+N}(z)&
=(\kappa-1)!s+\lambda_\kappa^{-\kappa-a}(z),
\end{align}
where $\phi _\kappa^s (N)=\kappa+s/N^\kappa$ and $\lim_{N\to \infty}\psi(N)=a$.

For $a+b \not =-1,-2,\cdots $ we write $(h_{n}^{a,b,N})_n$ for the sequence of Hahn polynomials defined by
\begin{equation}\label{hpol}
h_{n}^{a,b,N}(x)=\frac{(a+1)_n}{n!}\pFq{3}{2}{-n,-x,n+a+b+1}{a+1,-N}{1}
\end{equation}
(we have taken a slightly different normalization from the one used in \cite{KLS}, pp, 234-7 from where
the next formulas can be easily derived). Notice that if $N$ is not a positive integer, $h_{n}^{a,b,N}$ is always a polynomial of degree $n$,
and for $N$ a positive integer, $h_{n}^{a,b,N}$ is a polynomial of degree $n$ for $n=0,\cdots , N$.
For $N$ a positive integer and $-1<a,b$, the polynomials $h_n^{a,b,N}$, $n=0,\cdots, N$, are orthogonal with respect to the positive measure
\begin{equation}\label{hw}
\rho_{a,b,N}=\sum _{x=0}^N \binom{x+a}{x}\binom{b+N-x}{N-x}\delta_{x}.
\end{equation}
For $\alpha,\beta \in \RR , \alpha,\beta \not=-1,-2,\cdots$, we use the standard definition of the Jacobi polynomials $(P_{n}^{\alpha,\beta})_n$
\begin{equation}\label{defjac}
P_{n}^{\alpha,\beta}(x)=2^{-n}\sum _{j=0}^n \binom{n+\alpha}{j}\binom{n+\beta}{n-j}(x-1)^{n-j}(x+1)^{j}
\end{equation}
(see \cite{EMOT}, pp. 169-173 and also \cite{KLS}, pp. 216-221).

For $\alpha,\beta, \alpha+\beta \not =-1,-2,\cdots$,
they are orthogonal with respect to a measure $\mu _{\alpha,\beta}=\mu _{\alpha,\beta}(x)dx$, which it is positive only when
$\alpha ,\beta >-1$, and then
\begin{equation}\label{jacw}
\mu_{\alpha,\beta}(x) =(1-x)^\alpha (1+x)^{\beta}, \quad -1<x<1.
\end{equation}
We will  use the following formulas
\begin{align}\label{Lagder}
\frac{(2n+\alpha+\beta)P_n^{\alpha,\beta-1}(x)}{(n+\alpha+\beta)}&=P_n^{\alpha,\beta}(x)+\frac{(n+\alpha)P_{n-1}^{\alpha,\beta}(x)}{(n+\alpha+\beta)},
\\\label{Lagdere}
\frac{(2n+\alpha+\beta)P_n^{\alpha-1,\beta}(x)}{(n+\alpha+\beta)}&=P_n^{\alpha,\beta}(x)-\frac{(n+\beta)P_{n-1}^{\alpha,\beta}(x)}{(n+\alpha+\beta)},
\\\label{Lagder2}
\frac{(2n+\alpha+\beta-2)_3 P_n^{\alpha-1,\beta-1}(x)}{(n+\alpha+\beta)(n+\alpha+\beta-1)^2}&=\begin{vmatrix} P_n^{\alpha,\beta}(x)&P_{n-1}^{\alpha,\beta}(x)&P_{n-2}^{\alpha,\beta}(x)\\
\frac{(n+\beta-1)_2}{(n+\alpha+\beta-1)_2}&\frac{n+\beta-1}{n+\alpha+\beta-1}&1\\
\frac{(n+\alpha-1)_2}{(n+\alpha+\beta-1)_2}&\frac{-(n+\alpha-1)}{n+\alpha+\beta-1}&1
\end{vmatrix},\\\label{jaci}
\frac{\frac{n}{n+\beta}P_n^{\alpha-1,\beta}(x)+P_{n-1}^{\alpha-1,\beta}(x)}{(n+\alpha+\beta-1)_2}&=
\frac{\begin{vmatrix} P_n^{\alpha,\beta}(x)&P_{n-1}^{\alpha,\beta}(x)&P_{n-2}^{\alpha,\beta}(x)\\
\frac{(n+\beta-1)_2}{(n+\alpha+\beta-1)_2}&\frac{n+\beta-1}{n+\alpha+\beta-1}&1\\
1&\frac{-n}{n+\beta}&\frac{(n-1)_2}{(n+\beta-1)_2}
\end{vmatrix}}{(2n+\alpha+\beta)(2n+\alpha+\beta-2)}.
\end{align}
Identities (\ref{Lagder}) and (\ref{Lagdere}) can be found in \cite{EMOT}, p. 173, (35) and (36), respectively. (\ref{Lagder2}) and (\ref{jaci}) can be deduced from (\ref{Lagder}) and (\ref{Lagdere}) after some calculations.

One can get the Jacobi polynomials (\ref{defjac}) from the Hahn polynomials (\ref{hpol}) by setting $x\to (1-x)N/2$ and taking limit as $N\to \infty$; more precisely
\begin{equation}\label{blhj}
\lim_{N\to +\infty}h_n^{a,b,N}\left(\frac{(1-x)N}{2}\right)=P_n^{a,b}(x).
\end{equation}
(see \cite{KLS}, p. 207; note that we are using for
Hahn polynomials a different normalization to that in \cite{KLS}).
From the recurrence relation for Hahn and Jacobi polynomials, one can easily extend that limit: if $\lim_{N\to \infty}\phi_1(N)=a$
and $\lim_{N\to \infty}\phi_2(N)=b$ then
\begin{equation}\label{blhj2}
\lim_{N\to +\infty}h_n^{\phi_1(N),\phi_2(N),N}\left(\frac{(1-x)N}{2}\right)=P_n^{a,b}(x).
\end{equation}
\bigskip

If we consider positive measures, Krall Jacobi polynomials can be obtained by adding a Dirac delta at one or two of the end points of the interval of
orthogonality and assuming that the corresponding parameters are nonnegative integers. We consider each case separately.

\subsection{Adding a Dirac delta at $-1$}
We start adding a Dirac delta at $-1$. For $\kappa$ a positive integer, and $\alpha,u\in \RR$ with $\alpha>-1, u> 0$,  the orthogonal polynomials $(P_n^{\alpha,\kappa;u})_n$ with respect to the positive measure
\begin{equation}\label{jack}
(1-x)^\alpha(1+x)^{\kappa-1} +\kappa !u\delta_{-1},
\end{equation}
are eigenfunctions of a differential operator of order $2\kappa+2$ (see \cite{Z,du1}). Orthogonal polynomials $(P_n^{\alpha,\kappa;u})_n$ with respect to (\ref{jack}) can be expanded in terms of two consecutive Jacobi polynomials as follows (see Lemma A.9 in \cite{du1}; also \cite{Il1})
\begin{equation}\label{ejacp}
P_n^{\alpha,\kappa;u}(x)= P_n^{\alpha,\kappa}(x)+\frac{(n+\alpha) \left[2^{\alpha+\kappa}\Gamma(\kappa)+u(n+1)_\kappa(n+\alpha+1)_{\kappa}\right]}{(n+\alpha+\kappa)\left[2^{\alpha+\kappa}\Gamma(\kappa)
+u(n)_\kappa(n+\alpha)_{\kappa}\right]}P_{n-1}^{\alpha,\kappa}(x).
\end{equation}
For the limiting case $u=\infty$, the identity (\ref{jaci}) gives
\begin{equation}\label{jacii}
\frac{nP_n^{\alpha-1,\kappa;\infty}(x)}{(n+\alpha+\kappa-1)_2(n+\kappa)}=
\frac{\begin{vmatrix} P_n^{\alpha,\kappa}(x)&P_{n-1}^{\alpha,\kappa}(x)&P_{n-2}^{\alpha,\kappa}(x)\\
\frac{(n+\kappa-1)_2}{(n+\alpha+\kappa-1)_2}&\frac{n+\kappa-1}{n+\alpha+\kappa-1}&1\\
1&\frac{-n}{n+\kappa}&\frac{(n-1)_2}{(n+\kappa-1)_2}
\end{vmatrix}}{(2n+\alpha+\kappa)(2n+\alpha+\kappa-2)}.
\end{equation}

In order to get Krall Hahn polynomials, we have to apply a suitable Christoffel transform to the Hahn measure. Among the several possibilities described in \cite{du0} (see also \cite{ddI2}), we consider here the following one. For $\kappa$ and $N$ positive integers and $a$ and $b$ real numbers with $-1<a$ and $\kappa<b$, the orthogonal polynomials $h_n^{a,b,N;\kappa}$, $n=0,\cdots ,N$, with respect to the positive measure
\begin{equation}\label{hk}
\sum_{x=0}^N (N+b-\kappa+1-x)_\kappa\binom{a+x}{x}\binom{N+b-\kappa-1-x}{N-x}\delta_x,
\end{equation}
are eigenfunctions of a difference operator of order $2\kappa+2$.  Orthogonal polynomials $(h_n^{a,b,N;\kappa})_n$ with respect to (\ref{hk}) can be expanded in terms of two consecutive Hahn polynomials as follows (see Theorem 7.5 in \cite{du1})
\begin{equation}\label{ehp}
h_n^{a,b,N;\kappa }(x)=h_n^{a,b,N}(x)
+\frac{(n+a)R_{\kappa}^{-b,-a,a+b+N} (\theta_{-n-1}^{-a-b})}{(n+a+b)R_{\kappa}^{-b,-a,a+b+N} (\theta_{-n}^{-a-b})}
h_{n-1}^{a,b,N}(x),
\end{equation}
where $R_n^{a,b,N}$ are the dual Hahn polynomials (\ref{dhpol}) and the function $\theta_x^u$ is defined in (\ref{defth}).
The positivity of the measure (\ref{hk}) implies that $R_{\kappa}^{-b,-a,a+b+N} (\theta_{-n}^{-a-b})\not =0$, $n\in \NN$.

As happens with Krall Meixner and Krall Laguerre polynomials, we cannot get Krall Jacobi polynomials (\ref{ejacp}) from Krall Hahn polynomials (\ref{ehp}) taking limit as in the Askey scheme for getting Jacobi from Hahn polynomials (see (\ref{blhj})).
\begin{lemma}
\begin{equation}\label{lhj1}
\lim_{N\to\infty}h_n^{a,b,N;\kappa}\left(\frac{(1-x)N}{2}\right)=\begin{cases}\frac{(2n+a+b)
}{n+a+b}P_n^{a,b-1}(x),& \mbox{for $b\not=1,\cdots, \kappa$,}\\
P_n^{a,b;\infty}(x),& \mbox{for $b=1,\cdots, \kappa$,}
\end{cases}
\end{equation}
where $P_n^{a,b;\infty}$ is the degenerated case of (\ref{ejacp}) for $u=\infty$ defined in (\ref{jacii}).
If the parameters do not satisfy the assumptions $-1<a$ and $\kappa<b$, the positivity of the measure (\ref{hk}) is not guaranteed (this is the case when $b=1,\cdots, \kappa $); in that case we  implicitly assume
that $R_{\kappa}^{-b,-a,a+b+N} (\theta_{-n}^{-a-b})\not =0$, $n\in\NN$, so that the polynomials (\ref{ehp}) are well-defined.
\end{lemma}
\begin{proof}
Indeed, the first limit in (\ref{lhj1}) can be easily deduced  using (\ref{ehp}), (\ref{blhj}), (\ref{lhk}) and (\ref{Lagder}).

On the other hand, the second limit in (\ref{lhj1}) is an easy consequence of (\ref{ehp}), (\ref{blhj}), (\ref{lhk2}), (\ref{lth}) and (\ref{ejacp}).

\end{proof}

However, one can get Krall Jacobi polynomials (\ref{ejacp}) from Krall Hahn polynomials (\ref{ehp}) if in (\ref{lhj1}) we force $b$ to be a suitable function of $N$. More precisely, we have
\begin{lemma}
$$
\lim_{N\to \infty}h_n^{a,\kappa+\frac{2^{a+\kappa}}{uN^\kappa},N;\kappa}\left(\frac{(1-x)N}{2}\right)=P_n^{a,\kappa;u}(x)
$$
\end{lemma}

\begin{proof}
It can be easily deduced  using (\ref{ehp}), (\ref{blhj2}), (\ref{lhk3}), (\ref{lth}) and (\ref{ejacp}).

\end{proof}

The case of adding a Dirac delta at $1$ can be worked on taking into account the following symmetry: the polynomials
$$
\mathcal P_n^{\kappa,\beta;u}(x)=P_n^{\beta,\kappa;u}(-x)
$$
are orthogonal with respect to the Krall Jacobi measure
$(1-x)^{\kappa-1}(1+x)^\beta +\kappa !u\delta_{1}$.

\subsection{Adding Dirac deltas at $-1$ and $1$}
We now add Dirac deltas at $-1$ and $1$ to a Jacobi weight and assume that $\alpha$ and $\beta$ are positive integers.
For $\kappa,\sigma$ positive integers, and $u,v\in \RR$ with $u,v> 0$,  the orthogonal polynomials $(P_n^{\begin{subarray}{l}
       u,v\vspace{-.055cm}\\
       \kappa,\sigma
   \end{subarray}})_n$ with respect to the positive measure
\begin{equation}\label{jack2}
(1-x)^{\kappa-1}(1+x)^{\sigma-1} +\frac{\kappa!u}{2}\delta_1+\frac{\sigma !v}{2}\delta_{-1},
\end{equation}
are eigenfunctions of a differential operator of order $2\kappa+2\sigma+2$ (see \cite{koekoe2,Il1,ddI3}). Orthogonal polynomials $(P_n^{\begin{subarray}{l}
       u,v\vspace{-.055cm}\\
       \kappa,\sigma
   \end{subarray}})_n$ with respect to (\ref{jack2}) can be expanded in terms of three consecutive Jacobi polynomials as follows (see Theorem 1.1 in \cite{ddI3}; also \cite{Il1}) (it is now better to use determinantal notation)
\begin{equation}\label{ejacp2}
P_n^{\begin{subarray}{l}
       u,v\vspace{-.055cm}\\
       \kappa,\sigma
   \end{subarray}}(x)=\begin{vmatrix} P_n^{\kappa,\sigma}(x)&P_{n-1}^{\kappa,\sigma}(x)&P_{n-2}^{\kappa,\sigma}(x)\\
u
+\frac{T_{\sigma}^{\kappa}(n)}{(n+1)_{\kappa}}&\frac{un}{(n+\kappa)}
+\frac{T_{\sigma}^{\kappa}(n-1)}{(n+1)_{\kappa}}&\frac{u(n-1)_2}{(n+\kappa-1)_2}
+\frac{T_{\sigma}^{\kappa}(n-2)}{(n+1)_{\kappa}}\\
v
+\frac{T_{\kappa}^{\sigma}(n)}{(n+1)_{\sigma}}&\frac{-vn}{(n+\sigma)}
-\frac{T_{\kappa}^{\sigma}(n-1)}{(n+1)_{\sigma}}&\frac{v(n-1)_2}{(n+\sigma-1)_2}
+\frac{T_{\kappa}^{\sigma}(n-2)}{(n+1)_{\sigma}}
\end{vmatrix}
\end{equation}
where $T_{\sigma}^{\kappa}(n)=\frac{2^{\kappa+\sigma}\Gamma(\kappa)}{(n+\sigma+1)_{\kappa}}$.

The limiting case $u=v=\infty$ is given by
\begin{equation}\label{ejacp2i}
P_n^{\begin{subarray}{l}
       \infty,\infty \vspace{-.055cm}\\
       \kappa,\sigma
   \end{subarray}}(x)=\begin{vmatrix} P_n^{\kappa,\sigma}(x)&P_{n-1}^{\kappa,\sigma}(x)&P_{n-2}^{\kappa,\sigma}(x)\\
1&\frac{n}{(n+\kappa)}&\frac{(n-1)_2}{(n+\kappa-1)_2}
\\
1&\frac{-n}{(n+\sigma)}&\frac{(n-1)_2}{(n+\sigma-1)_2},
\end{vmatrix}, \quad n\ge 2
\end{equation}
with initial values $P_0^{\begin{subarray}{l}
       \infty,\infty \vspace{-.055cm}\\
       \kappa,\sigma
   \end{subarray}}(x)=1$ and
$$
P_1^{\begin{subarray}{l}
       \infty,\infty \vspace{-.055cm}\\
       \kappa,\sigma
   \end{subarray}}(x)=
\begin{vmatrix} \frac{1}{1+\kappa}&\frac{T^\kappa_\sigma (-1)}{(2)_\kappa}\\
\frac{-1}{1+\sigma}&\frac{T^\sigma_\kappa (-1)}{(2)_\sigma}\end{vmatrix}P_1^{\kappa,\sigma}(x)-
\begin{vmatrix}
1&\frac{T^\kappa_\sigma (-1)}{(2)_\kappa}\\
1&\frac{T^\sigma_\kappa (-1)}{(2)_\sigma}\end{vmatrix}.
$$

We now consider the following Christoffel transform of the Hahn measure (there are other possibilities: see \cite{du0} and also \cite{ddI2}). For $\kappa,\sigma$ and $N$ positive integers and $a$ and $b$ real numbers with $\kappa<a$ and $\sigma<b$, the orthogonal polynomials $h_n^{\begin{subarray}{l}
       \kappa,\sigma \vspace{-.055cm}\\
       a,b,N
   \end{subarray}}$, $n=0,\cdots ,N$, with respect to the positive measure
\begin{equation}\label{hk2}
\sum_{x=0}^N (a-\kappa+1+x)_{\kappa} (N+b-\sigma+1-x)_{\sigma}\binom{a-\kappa-1+x}{x}\binom{N+b-\sigma-1-x}{N-x}\delta_x,
\end{equation}
are eigenfunctions of a difference operator of order $2\kappa+2\sigma+2$.  Orthogonal polynomials $(h_n^{\begin{subarray}{l}
       \kappa,\sigma \vspace{-.055cm}\\
       a,b,N
   \end{subarray}})_n$ with respect to (\ref{hk2}) can be expanded in terms of three consecutive Hahn polynomials as follows (see Theorem 6.2 in \cite{ddI2})
\begin{equation}\label{ehp2}
h_n^{\begin{subarray}{l}
       \kappa,\sigma \vspace{-.055cm}\\
       a,b,N
   \end{subarray}}(x)=\begin{vmatrix} h_n^{a,b,N}(x)&h_{n-1}^{a,b,N}(x)&h_{n-2}^{a,b,N}(x)\\
\frac{(n+b-1)_2}{(n+a+b-1)_2}S^{b,a,N}_{\kappa}(n)&
\frac{n+b-1}{n+a+b-1}S^{b,a,N}_{\kappa}(n-1)&
S^{b,a,N}_{\kappa}(n-2)\\
\frac{(n+a-1)_2}{(n+a+b-1)_2}S^{a,b,N}_{\sigma}(n)&
-\frac{n+a-1}{n+a+b-1}S^{a,b,N}_{\sigma}(n-1)&
S^{a,b,N}_{\sigma}(n-2)
\end{vmatrix},
\end{equation}
where $S^{a,b,N}_\kappa(n)=R_\kappa^{-b,-a,a+b+N}(\theta_{-n-1}^{-a-b})$.
The positivity of the measure (\ref{hk2}) implies that the polynomial $h_n^{\begin{subarray}{l}
       \kappa,\sigma \vspace{-.055cm}\\
       a,b,N
   \end{subarray}}$ has degree $n$ for $n\in \NN$.

As happens with the other cases considered in this paper, we cannot get Krall Jacobi polynomials (\ref{ejacp2}) from Krall Hahn polynomials (\ref{ehp2}) taking limit as in the Askey scheme for getting Jacobi from Hahn polynomials (see (\ref{blhj})).

\begin{lemma} For $a\not=1,\cdots, \kappa$, and $b\not=1,\cdots, \sigma$, we have
\begin{equation*}\label{lhj2}
\lim_{N\to\infty}\frac{h_n^{\begin{subarray}{l}
       \kappa,\sigma \vspace{-.055cm}\\
       a,b,N
   \end{subarray}}((1-x)N/2)}{(a-\kappa)_{\kappa}(b-\sigma)_{\sigma}N^{\kappa+\sigma}}=\frac{
(2n+a+b-2)_3
}{(n+a+b)(n+a+b-1)^2}P_n^{a-1,b-1}(x).
\end{equation*}

For $a\not =1,\cdots, \kappa$, and $b=1,\cdots, \sigma$, we have ($n\ge 1$)
\begin{equation*}\label{xlhj2}
\lim_{N\to\infty}\frac{h_n^{\begin{subarray}{l}
       \kappa,\sigma \vspace{-.055cm}\\
       a,b,N
   \end{subarray}}((1-x)N/2)}{(a-\kappa)_{\kappa}N^{\kappa+\sigma-b}}=
\frac{n(2n+a+b-2)(2n+a+b)
}{(n+b)(n+a+b-1)_2}\Lambda_\sigma^{b,a}(n)P_n^{a-1,b;\infty}(x),
\end{equation*}
where $\Lambda_\sigma^{b,a}(n)=(-1)^{\sigma+b}(b+1)_{\sigma-b}(n+a-1)_b(n+1)_{b}$ and $P_n^{a,b;\infty}$ is the degenerated case of (\ref{ejacp}) for $u=\infty$.

For $a=1,\cdots, \kappa$, and $b\not=1,\cdots, \sigma$, we have ($n\ge 1$)
\begin{equation*}\label{ylhj2}
\lim_{N\to\infty}\frac{h_n^{\begin{subarray}{l}
       \kappa,\sigma \vspace{-.055cm}\\
       a,b,N
   \end{subarray}}((1-x)N/2)}{(b-\sigma)_{\sigma}N^{\kappa+\sigma-a}}=
\frac{n(2n+a+b-2)(2n+a+b)
}{(-1)^n(n+a)(n+a+b-1)_2}\Lambda_\kappa^{a,b}(n)P_n^{b-1,a;\infty}(-x).
\end{equation*}

For $a=1,\cdots, \kappa$, and $b=1,\cdots, \sigma$, we have
\begin{equation*}\label{lhj3}
\lim_{N\to\infty}\frac{h_n^{\begin{subarray}{l}
       \kappa,\sigma \vspace{-.055cm}\\
       a,b,N
   \end{subarray}}((1-x)N/2)}{N^{\kappa+\sigma-a-b}}=
\Lambda_{\kappa}^{a,b}(n)\Lambda_{\sigma}^{b,a}(n)P_n^{\begin{subarray}{l}
       \infty,\infty \vspace{-.055cm}\\
       a,b
   \end{subarray}}(x),
\end{equation*}
where $P_n^{\begin{subarray}{l}
       \infty,\infty \vspace{-.055cm}\\
       a,b
   \end{subarray}}$ is the degenerated case of (\ref{ejacp2}) for $u=v=\infty$.

If the parameters do not satisfy the assumptions $\kappa<a$ and $\sigma<b$,  the positivity of the measure (\ref{hk2}) is not guaranteed (this is
what happens in the last three cases);  we then  implicitly assume
that the polynomials (\ref{ehp2}) have degree $n$.
\end{lemma}

\begin{proof}
Indeed, the first limit in this Lemma can be easily deduced  using (\ref{ehp2}), (\ref{blhj}), (\ref{lhk}) and (\ref{Lagder2}).

The second and third limits can be deduced from (\ref{ehp2}), (\ref{blhj}), (\ref{lhk}), (\ref{lhk2}), (\ref{lth}) and (\ref{jacii}).

Finally, the last limit  is an easy consequence of (\ref{ehp2}), (\ref{blhj}), (\ref{lhk2}), (\ref{lth}) and (\ref{ejacp2}).

\end{proof}

However, one can get Krall Jacobi polynomials (\ref{ejacp2}) from Krall Hahn polynomials (\ref{ehp2}) if  we force $a$ and $b$ to be suitable functions of $N$. More precisely, we have.
\begin{lemma}
$$
\lim_{N\to \infty}\frac{uvh_n^{\begin{subarray}{l}
       \kappa,\sigma\vspace{-.13cm}\\
\kappa+\frac{2^{\kappa+\sigma}}{uN^\kappa},\sigma+\frac{2^{\kappa+\sigma}}{vN^\sigma},N
   \end{subarray}}
\left(\frac{(1-x)N}{2}\right)}{(n+1)_\kappa(n+\sigma-1)_{\kappa}(n+1)_\sigma(n+\kappa-1)_{\sigma}}=
P_n^{\begin{subarray}{l}
       u,v\vspace{-.055cm}\\
       \kappa,\sigma
   \end{subarray}}(x).
$$
\end{lemma}

\begin{proof}
It can be easily deduced  using (\ref{ehp2}), (\ref{blhj2}), (\ref{lhk3}), (\ref{lth}) and (\ref{ejacp2}).
\end{proof}

\section*{References}



\end{document}